\documentclass{amsart}

\usepackage{amsmath, amsthm,amssymb, amsfonts}
\usepackage{rotating}
\usepackage{url}
\usepackage{microtype}
\usepackage{cite}

\renewcommand{\phi}{\varphi}

\newcommand{\comment}[1]{}

\def\Z{{\mathbb Z}}

\def\Q{{\mathbb Q}}

\def\P{{\mathbb P}}
\def\Zh{{\widehat{\Z}}}

\def\rank{{\operatorname{rank}}}

\def\Aut{{\operatorname{Aut \;}}}

\def\coker{{\operatorname{coker}}}

\newtheorem*{Berry-Esseen}{Berry-Esseen Theorem}

\newtheorem{thm}{Theorem}
\newtheorem{conj}{Conjecture}

\newtheorem{cor}[thm]{Corollary}
\newtheorem{lemma}[thm]{Lemma}
\newtheorem*{defn}{Definition}

\theoremstyle{remark}

\title{The corank of a rectangular random integer matrix.}

\author{Shaked Koplewitz}
\address{Mathematics Department, Yale University, New Haven, CT 06511}
\email{shaked.koplewitz@gmail.com}

\subjclass{05,15}
\begin{document}

\begin{abstract}
We show that under reasonable conditions, a random $n\times (2+\epsilon) n$ integer matrix is surjective on $\Z^{n}$ with probability $1-O(e^{-cn})$. We also conjecture that this should hold for $n\times (1+\epsilon)n$, and provide a counterexample to show that our ``reasonableness'' conditions are necessary.

\end{abstract}

\maketitle

\section{Introduction}

In \cite{bvw}, Bourgain, Vu, and Wood show that, given an $n \times n$ random matrix $A$ whose entries take the values $+1,-1$ independently with probability $\frac{1}{2}$, the probability that $A$ is singular is bounded by $(\frac{1}{\sqrt{2}}+o(1))^{n}$. In particular, this implies that $A$ is injective (as a map $A:\Z^{n}\rightarrow\Z^{n}$) with probability $1-O(e^{-cn})$ for some constant $c>0$. In this paper, we ask:

\textbf{Question} Let $A:\Z^{m}\rightarrow\Z^{n}$ be a random integer matrix (for $m\geq n$). What is the probability that $A$ is surjective?

Our main result is the following:
\begin{thm}
\label{thm:main}
Let $A$ be an $\epsilon$-balanced $n\times (2+\delta)n$ random matrix with entries $|A_{ij}|=O(2^{n^{k}})$ for some constant $k$. Then $A$ is surjective with probability $1-O(e^{-cn})$ for some constant $c>0$ as $n\rightarrow\infty$.
\end{thm}
\noindent We recall the definition of $\epsilon$-balanced in Section~\ref{sec:constDiff}.

Some type of independence assumption like $\epsilon$-balancedness is clearly necessary to avoid trivial counterexamples, such as the entries of $A$ all being equal with probability $1$. We will also show in Section~\ref{sec:counterexample} that the bound on the size is also necessary, for any $m$, by giving a counterexample when the entries are allowed to be of size up to $e^{3^{nm}}$.

We also show the following holds, as a direct consequence of the results of Wood in \cite{wr}:

\begin{thm}
\label{thm:constDiff}
Let $A$ be an $\epsilon$-balanced random $n\times (n+u)$ matrix, with $\epsilon$ and $u\geq 1$ constants, then 
\[\limsup_{n\rightarrow\infty}\P(A\text{ is surjective})\leq \prod_{p\text{ prime}}\prod_{k=1}^{\infty}(1-p^{-k-u})=\prod_{k=u+1}^{\infty}\zeta(k)^{-1}.\]
If $u=0$, then $\lim_{n\rightarrow\infty}\P(A\text{ is surjective})=0$.
\end{thm}

\noindent In particular, both results hold for $0$-$1$ Bernoulli random matrices, in which the entries are independently chosen to be $1$ with probability $q$ and $0$ otherwise, for constant $0<q<1$.

We conjecture that under the conditions of Theorem~\ref{thm:constDiff}, $\lim_{n\rightarrow\infty}\P(A\text{ is surjective})=\prod_{k=u+1}^{\infty}\zeta(k)^{-1}$. In particular, we guess the following:

\begin{conj}
\label{conj:surj}
Let $A$ be an $\epsilon$-balanced $n\times (1+\delta)n$ random matrix for constant $\epsilon,\delta>0$, with entries bound by $n^{k}$ for some constant $k>0$. Then $\lim_{n\rightarrow\infty}\P(A\text{ is surjective})=1$.
\end{conj}

Since the time of the original writing, this conjecture has been proved by Nguyen and Wood in \cite{nw}. See section~\ref{sec:further} for some of their results.

Finally, we ask what we can prove under stronger assumptions. The strongest possible case would be when the entries of the matrix are `uniformly distributed' in $\Z$. However, there is no uniform distribution over $\Z$. Our approach to resolving this is to use the Haar measure over the profinite completion $\Zh$, which will give us the following theorem:

\begin{thm}
\label{thm:compsurj}
Let $u\geq 0$ be constant, and let $A:\Zh^{n+u}\rightarrow\Zh^{n}$ be a random matrix, whose entries are independent identically distributed random variables given by the Haar measure on $\Zh$. Then if $u>0$, 
\[\lim_{n\rightarrow\infty}\P(A\text{ is surjective})=\prod_{k=u+1}^{\infty}\zeta(k)^{-1}.\] 

\noindent If $u=0$, this probability converges to zero.
\end{thm}

In particular, Theorem~\ref{thm:compsurj}, along with the observation that the probability that $A$ is surjective monotonically increases with $u$, implies the following corollary:

\begin{cor}
\label{cor:compsurj}
Let $u(n)$ be a sequence such that $\lim_{n\rightarrow\infty}u(n)=\infty$, and let $A:\Zh^{n+u(n)}\rightarrow\Zh^{n}$ be a random matrix, whose entries are independent identically distributed random variables given by the Haar measure on $\Zh$. Then
\[\lim_{n\rightarrow\infty}\P(A\text{ is surjective})=1\]
\end{cor}

\noindent In particular, this implies that a random $n\times cn$ matrix over $\Zh$ with $c>1$ will be surjective with probability $\rightarrow 1$.

Another natural approach is to take $A=A_{n,m,k}$ to be the matrix whose entries are independent identically distributed random variables uniformly distributed in ${-k,\dots,k}$, and take $k\rightarrow\infty$. The authors of \cite{mn} show that
\[\lim_{k\rightarrow\infty}\P(A_{n,m,k}\text{ is surjective})=\P(A_{n,m}\text{ is surjective}),\]
where $A_{n,m}$ is a random $n\times m$ matrix over $\Zh$.

\textbf{Acknowledgements}.
The author is  grateful to Sam Payne, Nathan Kaplan, and Van H. Vu for their many helpful suggestions along the way.

This work was partially supported by NSF CAREER DMS-1149054.

\section{Results for $n\times (n+u)$ Matrices for constant $u$}
\label{sec:constDiff}

In this section, we prove Theorem~\ref{thm:constDiff}. We will rely on the following lemma:
\begin{lemma}
\label{lemma:places}
A matrix $A:\Z^{m}\rightarrow\Z^{n}$ is surjective if and only if $A/p:(\Z/p\Z)^{m}\rightarrow(\Z/p\Z)^{n}$ is surjective for every prime $p$. Here $A/p$ is the matrix over $\Z/p\Z$ given by $(A/p)_{ij}=A_{ij}\pmod{p}$. 
\end{lemma}

\begin{proof}
Clearly, if $A:\Z^{m}\rightarrow\Z^{n}$ is surjective, so is $A/p:(\Z/p\Z)^{m}\rightarrow(\Z/p\Z)^{n}$ for every $p$.

Conversely, assume $A/p:(\Z/p\Z)^{m}\rightarrow(\Z/p\Z)^{n}$ is surjective for every $p$. Then $A/p$ has rank $n$, and in particular contains an $n\times n$ submatrix $B/p\subseteq A/p$ with nonzero determinant. As $\det(B/p)=\det(B)\pmod{p}$, this implies that $\det(B)$ is nonzero. This implies that the columns of $B$, considered as vectors over $\Q$, generate $\Q^{n}$ as a vector space, and hence $B(\Z^{n})$ is a full-rank lattice. But $A(\Z^{m})$ is an abelian group containing $B(\Z^{n})$, so it must also be a full-rank lattice. In particular, $D=|\Z^{n}/A\Z^{m}|$ is finite, and $D$ divides $|\det(B)|$.

But $p\nmid|\det(B)|$, hence $p\nmid D$. As this hold for every prime $p$, we get $D=1$, so $\Z^{n}=A\Z^{m}$, which completes the proof.
\end{proof}
\noindent In particular, this theorem implies that a square matrix is surjective if it is nonsingular at every prime $p$. In contrast, a square matrix is injective if it is nonsingular at any prime $p$.

It is worth noting that the random matrices of \cite{bvw}, whose entries are $\pm 1$, are never surjective, since $A/2$ is the all-ones matrix.

We now recall the following definition from \cite{wr}:
\begin{defn}
A random variable $y$ taking values in a ring $T$ is $\epsilon$-balanced if for every maximal ideal $\mathfrak{p}$ of $T$ and every $r\in T/\mathfrak{p}$,we have $\P(y\equiv r\pmod{\mathfrak{p}})\leq (1-\epsilon)$. In particular, if $T$ is a field, $y$ is $\epsilon$-balanced if for every $r\in T$, we have $\P(y=r)\leq (1-\epsilon)$.

A random matrix is $\epsilon$-balanced if its entries are independent and $\epsilon$-balanced.
\end{defn}
\noindent In particular, a matrix whose entries are independent identically distributed Bernoulli random variables equal to $0$ with probability $1>q>0$ and $1$ otherwise is $\epsilon$-balanced, for $\epsilon=\min(q,1-q)$.

For any abelian group $G$ and prime $p$, define $G_{p}$ to be a $p$-Sylow subgroup. If $P$ is a set of primes, define $G_{P}=\prod_{p\in P}G_{p}$. We now recall the following theorem of Wood:

\begin{thm}[Corollary 3.4 of \cite{wr}]
\label{thm:wr}
Let $\epsilon>0$ and let $A$ be an $\epsilon$-balanced $n\times(n+u)$ random matrix. Let $G$ be a finite abelian group and let $P$ be a finite set of primes including all those dividing $|G|$. Then:

\[ \lim_{n\rightarrow\infty}\P((\Z^{n}/(A\Z^{n+u}))_{P}\simeq G)=\frac{1}{|G|^{u}|\Aut(G)|}\prod_{p\in P}\prod_{k=1}^{\infty}(1-p^{-k-u}).\]
\end{thm}

Using this, we now prove Theorem~\ref{thm:constDiff}.

\begin{proof}[Proof of Theorem~\ref{thm:constDiff}]
Let $A:\Z^{n+u}\rightarrow\Z^{n}$ be an $\epsilon$-balanced random matrix. By Lemma~\ref{lemma:places}, $A$ is surjective only if $A/p$ is surjective for every prime $p$. This is equivalent to $(\Z^{n}/(A\Z^{n+u}))_{p}$ being the trivial group for every prime $p$.

Let $P$ be a finite set of primes. Then by Theorem~\ref{thm:wr} with $G=1$, 
\[ \lim_{n\rightarrow\infty}(\P((\Z^{n}/(A\Z^{n+u}))_{p}\simeq 1\text{ for all }p\in P)=\prod_{p\in P}\prod_{k=1}^{\infty}(1-p^{-k-u}).\]

But for any finite set $P$, this is an upper bound on $\limsup\P(A\text{ is surjective})$. Taking $P$ to be increasingly large gives us the theorem.
\end{proof}

\section{Surjectivity of random $n\times (2+\delta)n$ matrices}
\label{sec:surj}

In this section, we prove Theorem~\ref{thm:main}. First, we recall the following theorem from \cite{skb}:

\begin{thm}
\label{thm:rectrank}
Let $A$ be an $\epsilon$-balanced $n\times m$ random matrix over a field $\mathbb{F}_p$ with $m\geq (1+\delta)n$ for some constant $\delta>0$. Then $A$ has full rank with probability at least $1-e^{-cn}$ for some constant $c$ depending only on $\epsilon,\delta$. In particular, $c$ is independent of $\mathbb{F}_p$.
\end{thm}

\noindent The proof bounds the probability that each row is dependent on the previous rows, similarly to the proof of Lemma~\ref{lemma:padicpart}.

In particular, This implies the following corollary:
\begin{cor}
\label{cor:det}
Let $A$ be an $\epsilon$-balanced $n\times m$  random matrix over $\Z$ with $m\geq (1+\delta)n$ for some constant $\delta>0$. Then there exists a constant $c$ depending only on $\epsilon,\delta$, such that with probability $\geq 1-e^{-c\epsilon n}$, $A$ contains an $n\times n$ submatrix $A'$ with nonzero determinant.
\end{cor}

\begin{proof}
Using Theorem~\ref{thm:rectrank} for $A/2$ gives us that
 \[\P(\rank{\left (A/2 \right)}=n)\geq 1-e^{-c\epsilon n}.\]

But if $A/2$ has rank $n$, it must contain a full rank $n\times n$ submatrix $A'/2$. Hence $\det(A')\neq 0\pmod{2}$, and in particular $\det(A')\neq 0$. 
\end{proof}

We are now ready to prove Theorem~\ref{thm:main}.

\begin{proof}[Proof of Theorem~\ref{thm:main}]
Let $A$ be a random $n\times(2+\delta)n$ matrix. We split it into two submatrices $A=\left(B,C\right)$, where $B$ and $C$ are $n\times (1+\frac{\delta}{2})n$. Note that $B$ and $C$ are both $\epsilon$-balanced.

As $B$ is $\epsilon$-balanced, by Corollary~\ref{cor:det} with probability at least $1-e^{-cn}$, it contains a submatrix $B'$ such that $\det(B')\neq 0$, where $c$ depends only on $\epsilon,\delta$.

We can also bound the size of $|\det(B')|$. Recall that the entries of $A$, hence in particular the entries $B'_{ij}$ of $B'$, are all bounded by $O(2^{n^{k}})$ for some constant $k$. Assume that $|B'_{ij}|\leq 2^{n^{k}}$. As the determinant is the sum of $n!$ products of permutations of the $B'_{ij}$, we can bound $|\det(B')|\leq (n!)(2^{n^{k}})^{n}\leq 2^{n^{k'}}$ for a sufficiently large constant $k'$. If $|\det(B')|$ is nonzero, the number of prime divisors of $|\det(B')|$ is bounded by $\log_{2}(|\det(B')|)\leq n^{k'}$. 

Let $P$ denote the set of prime divisors of $|\det(B')|$. By Theorem~\ref{thm:rectrank}, for every $p\in P$, $C/p$ is surjective over $\Z/p\Z$ with probability at least $1-e^{-c'n}$ for some $c'$ depending only on $\epsilon,\delta$. If $\det(B')$ is nonzero, then $|P|\leq n^{k'}$, so the probability that $C/p$ is surjective over every $p\in P$ is at least $1-n^{k'}e^{-c'n}$. As $\det(B')$ is nonzero with probability at least $1-e^{-c'n}$, the probability that $C/p$ is surjective for every $p\in P$ is at least $1-(n^{k'}+1)e^{-c'n}\geq 1-e^{-an}$ for some constant $a>0$.

But if this holds, then $A$ is surjective: $B'/p$ is surjective over every $p\notin P$, and $C/p$ is surjective over $p\in P$. Hence $A/p$ is surjective over every prime $p$, so by Lemma~\ref{lemma:places}, $A$ is surjective.
\end{proof}

\section{Counterexample with large entries}
\label{sec:counterexample}
In this section, we show that the bound on the size of the entries given in Theorem~\ref{thm:main} is necessary by showing a distribution of the entries with size bounded by $e^{3^{nm}}$, where the probability that a matrix is surjective goes to zero (In fact, the entries will be bounded by $e^{n^{2}\log(n)m2^{nm}}$). Note that this depends on $m$, so taking $m$ to be a large function of $n$ cannot resolve this need for a bound on the size of the entries.

Let $P$ be the set of the first $2^{nm}n$ primes. For each $i,j$ we choose a subset $P'_{i,j}\subseteq P$ independently at random by taking $p\in P'_{i,j}$ independently with probability $\frac{1}{2}$ for every $p\in P$. We let $A_{ij}=\prod_{p\in P'_{i,j}}p$.

$A$ is $\epsilon$-balanced for $\epsilon=\frac{1}{2}$: At a prime $p\in P$, this is obvious, since $\P(A_{ij}\equiv 0\pmod{p})=\frac{1}{2}$. for $p\notin P$, this follows by noting that when we choose whether to put the last prime of $P$ in $P'$, we choose whether or not to change $A_{ij}\pmod{p}$ with probability  $\frac{1}{2}$.

To see the bound on the size of $A_{ij}$, note that $A_{ij}$ is bounded by the product of the first $2^{nm}n$ primes. In general, the product of the first $k$ primes is bounded by $e^{2k\log(k)}$ (see for example \cite{ba}). Taking $k=2^{nm}n$, we see that 
\[|A_{ij}|\leq e^{2^{nm+1}n\log(2^{nm}n)}\leq e^{n^{2}\log(n)m2^{nm+1}}\leq e^{3^{nm}}\]
for all sufficiently large $n$.

Finally, for every $p\in P$, If all the entries of $A$ are zero mod $p$, then $A$ is not surjective. For each $p\in P$, this occurs independently with probability $2^{-nm}$. As there are $2^{nm}n$ primes in $P$, the probability of being surjective is at most $(1-2^{-nm})^{2^{nm}n}=((1-2^{-nm})^{2^{nm}})^{n}\leq e^{-n}\rightarrow 0$. This shows that the conclusion of Theorem~\ref{thm:main} does not hold in this case.

\section{Random matrices over $\Zh$}
\label{sec:zhat}

In this section, we prove Theorem~\ref{thm:compsurj}.

First, recall that $\Zh=\prod_{p}\Z_{p}$, and that the Haar measure on $\Zh$ is the product of the Haar measures on $\Z_{p}$. Furthermore, the matrix $A:\Zh^{m}\rightarrow\Zh^{n}$ is the direct product of the matrices $A_{p}:\Z_{p}^{m}\rightarrow\Z_{p}^{n}$, Where $A_{p}$ is the $n\times m$ matrix given by taking the $\Z_{p}$-part of the coefficients of $A$. In particular, $A$ is surjective only if $A_{p}$ is surjective for every $p$. 

When the entries of $A$ are independent random variables in $\Zh$, the $A_{p}$ are all independent random matrices whose values are independent uniformly distributed random variables in $\Z_{p}$. Hence 
\[\P(A\text{ is surjective})=\prod_{p}\P(A_{p}\text{ is surjective}).\]
The main part of the proof of Theorem~\ref{thm:compsurj} is the following lemma:

\begin{lemma}
\label{lemma:padicpart}
Let $A_{p}$ be a random $n\times m$ matrix, with $m\geq n$, whose entries are independent and uniformly distributed in $\Z_{p}$. Then
\[\P(A_{p}\text{ is surjective})=\prod_{k=m+1-n}^{m}(1-p^{-k}).\]
\end{lemma}

\begin{proof}
First, recall that $A_{p}$ is surjective on $\Z_{p}^{n}$ only if $A_{p}/p$ is surjective on $\Z_{p}/p\Z_{p}^{n}=\Z/p\Z^{n}$. Hence we can consider $A_{p}/p$. Note that its entries are independent and uniformly distributed in $\Z/p\Z$. 

As a matrix over a field, $A_{p}/p$ is surjective only if it has rank $n$, which happens only if its $n$ rows are independent. We prove by induction that the probability of the first $r$ rows being independent is $\prod_{k=m+1-r}^{m}(1-p^{-k})$.

Let $u_{1},\dots,u_{n}$ be the rows of $A_{p}/p$. For $r=1$, $u_{1}$ is independent only if it is nonzero. As it has $m$ independent entries, this happens with probability $1-p^{-m}$.
Now assume the claim for $r$. The first $r+1$ rows, $u_{1},\dots,u_{r+1}$ are independent only if $u_{1},\dots,u_{r}$ are independent and $u_{r+1}$ is independent of them. By the assumption, the probability that the first $u_{1},\dots,u_{r}$ rows are independent is $\prod_{k=m+1-r}^{m}(1-p^{-k})$. 
If the first $u_{1},\dots,u_{r}$ rows are independent, there exists some set $I$ of $r$  columns such that $u_{1}|_{I},\dots,u_{r}|_{I}$ are independent. Then there are unique coefficients $a_{1},\dots,a_{r}$ such that $u_{r+1}|_{I}=\sum a_{i}u_{i}|_{I}$, and $u_{r+1}$ is dependent on $u_{1},\dots,u_{r}$ only if $(u_{r+1})_{j}=\sum a_{i}(u_{i})_{j}$ for every $j\notin I$. Since $(u_{r+1})_{j}$ is independent of the rest of the matrix, this happens with probability $\frac{1}{p}$. As there are $m-r$ values for $j\notin I$, this implies that the probability that $u_{r+1}$ is dependent on $u_{1},\dots,u_{r}$ is $p^{-(m-r)}$. Hence overall, the probability that $u_{1},\dots,u_{r+1}$ are independent is 
\[(1-p^{-(m-r)})\prod_{k=m+1-r}^{m}(1-p^{-k})=\prod_{k=m+1-(r+1)}^{m}(1-p^{-k}),\]

\noindent which completes the proof.
\end{proof}

We now prove Theorem~\ref{thm:compsurj}.
\begin{proof}[Proof of Theorem~\ref{thm:compsurj}]
Let $A$ be a random $n\times(n+u)$ matrix over $\Zh$. As we saw, 
\begin{align*}
\P(A\text{ is surjective})&=\prod_{p}\P(A_{p}\text{ is surjective}) \\
&=\prod_{p}\prod_{k=u+1}^{n+u}(1-p^{-k}) \\
&=\prod_{k=u+1}^{n+u}\prod_{p}(1-p^{-k}) \\
&=\prod_{k=u+1}^{n+u}\zeta(k)^{-1}\rightarrow\prod_{k=u+1}^{\infty}\zeta(k)^{-1}. 
\end{align*}
The last two lines hold when $u>0$. When $u=0$, $\prod_{p}(1-p^{-1})=0$, so the product converges to zero.
\end{proof}

\section{Further results}
\label{sec:further}

Since the original writing of this paper, a number of the results conjectured here have been proven. In \cite{nw}, Nguyen and Wood prove Conjecture~\ref{conj:surj}: 

\begin{thm}[Theorem 1.4 of\cite{nw}]
\label{thm:nw_surj}
For integers $n,u\geq 0$, let $A_{n\times(n+u)} $ be an integral $n\times(n+u)$ matrix with entries i.i.d copies
of an $\alpha_{n}$-balanced random integer $\xi_{n}$, with $\alpha_{n}\geq n^{-1+\epsilon}$ and $|\xi_{n}| \leq n^{T}$ for any fixed parameters $0<\epsilon<1$ and $T>0$ not depending on $n$. Then 
\[\lim_{\min{n,u}\rightarrow\infty}\P(A_{n\times(n+u)}\text{ is surjective}) = 1.\]
\end{thm}

In fact, they prove stronger results. Their main theorem is the following:
\begin{thm}[Theorem 1.1 of\cite{nw}]
\label{thm:nw_main}
Let $n,u,A_{n\times(n+u)}$ be as above, and let $B$ be a fixed finite abelian group. The for any fixed $u$,
\[\lim_{n\rightarrow\infty}\P\left (\coker{A_{n\times(n+u)}}\simeq B\right ) = \frac{1}{\left | B \right | \left | \Aut{(B)}\right |}\prod_{k=u+1}^{\infty} \zeta(k)^{-1}.\]
\end{thm}

And these two imply the following corollary:
\begin{cor}[Theorem 1.5 of \cite{nw}]
\label{cor:nw_surj}
Let $n,u,A_{n\times(n+u)}$ be as above. For any fixed $u \geq 0$:
\[\lim_{n\rightarrow\infty}\P(A_{n\times(n+u)}\text{ is surjective}) = \prod_{k=u+1}^{\infty} \zeta(k)^{-1}.\]
\end{cor}

Finally, in \cite{np}, Nguyen and Paquette give some bounds on the speed of convergence for the probability of the matrix being surjective. In particular, they prove:
\begin{thm}[Theorem 1.4 (5) of\cite{np}]
\label{thm:np_main}
Let $n,u,A_{n\times(n+u)}$ be as above, and assume that $\alpha_{n}\geq\frac{\log^{O(1)}(n)}{n}$. Then
\[ \P\left(A_{n\times(\left \lfloor (1+o(1))n \right \rfloor)}\text{ is surjective}\right) = 1 - O\left ( n^{-\omega (1)}\right ).\]
\end{thm}

\bibliography{inMatrix}
\bibliographystyle{plain}

\end{document}